\documentclass{amsart}

\usepackage{amsmath}%
\usepackage{amsfonts}%
\usepackage{amssymb}%
\usepackage{graphicx}

\usepackage{todonotes}

\usepackage{amsthm}
\usepackage{graphicx}

\usepackage{mathrsfs}
\usepackage[all]{xy}
\usepackage{hyperref}

\newtheorem{theorem}{Theorem}
\theoremstyle{plain}

\newtheorem{corollary}[theorem]{Corollary}
\newtheorem{definition}[theorem]{Definition}

\newtheorem{lemma}[theorem]{Lemma}

\newtheorem{proposition}[theorem]{Proposition}

\numberwithin{equation}{section}

\newcommand{\PP}{\mathbb{P}}
\newcommand{\QQ}{\mathbb{Q}}
\newcommand{\FF}{\mathbb{F}}

\DeclareMathOperator{\Branch}{Branch}

\DeclareMathOperator{\Aut}{Aut}
\newcommand{\sh}{\mathscr}
\newcommand{\ZZ}{\mathbb{Z}}
\newcommand{\CC}{\mathbb{C}}

\newcommand{\J}{\mathcal{J}}

\begin{document}

\title{Strongly Cyclic Coverings of Cyclic Curves}
\author{Charles Siegel}
\address[Charles Siegel]{Kavli Institute for the Physics and Mathematics of the Universe (WPI), Todai Institutes for Advanced Study, the University of Tokyo}
\email[Charles Siegel]{charles.siegel@ipmu.jp}%
\urladdr{http://db.ipmu.jp/member/personal/2754en.html}
\date{}
\subjclass{} %
\keywords{}%

\begin{abstract}
In this note, we introduce the notion of an unramified strongly cyclic covering for a cyclic curve, a class that has similar properties to, and contains, unramified double covers of hyperelliptic curves.  We determine several of their basic properties, extending the theorems in \cite{MR2217998} to this larger class.  In particular, we will write down equations for smooth affine models, determine when they are isomorphic, and discuss the curves that they are ramified cyclic covers of.
\end{abstract}
\maketitle

The algebra and geometry of double covers of hyperelliptic curves is well understood, and has been studied by, among many others, Farkas \cite{MR0437741,MR922295,MR908651} and Bujalance \cite{MR855143}, from the topological viewpoint and Fuertes and Gonz\'alez-Diez \cite{MR2217998} from a more algebraic standpoint.  This can be done because the square-trivial line bundles on the curve can all be described in terms of the fixed points of the hyperelliptic involution.  This yields concrete equations, which can then be used to get very strong geometric information about the covers that cannot be so readily obtained for a general curve.

One would like to generalize these results to understand concretely examples of higher degree cyclic coverings, as studied by Faber in \cite{MR954752} and recently by Lange and Ortega in \cite{MR2753712}.  The natural generalization is to cyclic curves of degree $d$, because for any fixed point of the cyclic automorphism, $p$, the class $dp$ does not depend on which fixed point was chosen.  Unfortunately, there are not enough of these points to describe every degree $d$ cyclic covering.  Below, we will study the set of cyclic coverings that can be written in terms of the fixed points, and will prove analogues of the results for double covers of hyperelliptic curves in this case.

Below, we will not distinguish between an affine algebraic curve, its smooth and complete model, and its model as a compact Riemann surface.

\begin{definition}[Cyclic Curve]
A cyclic curve is a curve $C$ with a morphism $f:C\to \PP^1$ of degree $d$ and Galois group $\ZZ/d\ZZ$ and $\ZZ/d\ZZ\mathrel{\lhd}\Aut(C)$.
\end{definition}

Our goal will be to study a class of divisors on a cyclic curve.  For the remainder of this note, let $f:C\to\PP^1$ be the cyclic $g^1_d$ on a curve of genus $g$, and $x_1,\ldots,x_r$ the ramification points of $f$.  We want to look at the divisors of the form $\sum_{i=1}^r n_ix_i$ and of degree $0$.  Equivalently, we can write them as $\sum_{i=1}^r n_ix_i-\frac{\sum n_i}{d}g^1_d$ with $0\leq n_i<d$, and we will equivalent mean the divisor of degree $0$ and the positive part by itself.

From here on, we fix the degree $d$ to be a prime number, and we will assume that the map $f$ is unramified over $\infty$, and we will identify $\PP^1\setminus\{\infty\}$ with the complex numbers.

\begin{proposition}
\label{prop:numstrongpoints}
There are $d^{r-2}$ points of order $d$ in $\J(C)[d]$ that can be written in the form $\sum n_ix_i$ as above.
\end{proposition}

\begin{proof}
Following \cite{MR2964027}, we let $R=\{x_1,\ldots,x_r\}$ be the set of ramification points.  Then we can look at $\FF_d^r=\{\mbox{the set of functions }R\to \FF_d\}$.  The hyperplane given by $\{\mu:R\to\FF_d|\sum_{i=1}^r \mu(x_i)=0\}$ gives the set of degree $0$ divisors.  The curve can be written $y^d=\prod_{i=1}^r (x-b_i)^{d_i}$, where $b_i=f(x_i)$, for some set of $1\leq d_i<d$ such that $\sum d_i\equiv 0$ modulo $d$.  We can see that this is trivial by comparing the lines $y=0$ and $z=0$, and so, finally, we quotient by this relation, and get $\FF_d^{r-2}$ which embeds into $\J(C)[d]$, the points of order $d$ on the Jacobian of $C$.
\end{proof}

Elements $\alpha\in \FF_d^{r-1}$ in the proof of Proposition \ref{prop:numstrongpoints} will also be identified with their images in $\J(C)[d]$, and as divisors $\sum_{i=1}^r \alpha(x_i)x_i$, with $\alpha(x_i)$ interpreted as an integer $0\leq \alpha(x_i)<d$.

\begin{definition}[Strongly Cyclic Cover]
The unramified cyclic cover corresponding to a point above will be called a strongly cyclic cover.  We will denote this cover by $\pi:\tilde{C}\to C$, and when context requires that we be more specific, we will use subscripts.
\end{definition}

This implies that only a small fraction of all cyclic covers are strongly cyclic.  For instance, if $g(C)=2$ and $d=3$, then there will be four ramification points and so there are at most $3^2$ points of order 3 that are strongly cyclic, but there are $3^4$ points of order $3$ on the Jacobian of $C$.

This is in contrast with the $d=2$ case.  For hyperelliptic curves, there are $2g+2$ ramification points, so there are $2^{2g}$ points of this form, and so every point of order two on a hyperelliptic Jacobian is one of these.

The Riemann-Hurwitz theorem applied to $\tilde{C}\stackrel{\pi}{\to}C\stackrel{f}{\to}\PP^1$ immediately gives us

\begin{lemma}
\label{lemma:numerics}
If $g$ is the genus of $C$ and $\tilde{g}$ is the genus of $\tilde{C}$, then

\begin{eqnarray*}
 g & = & \frac{(r-2)(d-1)}{2}\\
\tilde{g} & = & \frac{(d-1)(rd-2d-2)}{2}\\
r & = & \frac{2g}{d-1}+2
\end{eqnarray*}
\end{lemma}

\begin{theorem}
\label{thm:equations}
Let $C$ be a cyclic curve and $\alpha\in\FF_d^{r-1}$ a nonzero element of the kernel of the map to $\J(C)[d]$.  Then $C$ is given by the equation $y^d=\prod_{i=1}^r (x-b_i)^{\alpha(x_i)}$.  Let $\beta\in \FF_d^{r-1}$ give a curve $C_\beta$ with equation $z^d=\prod_{i=1}^r (x-b_i)^{\beta(x_i)}$.  Then:

\begin{enumerate}
	\item the normalization of $\tilde{C}_\beta=C\times_{\PP^1} C_\beta$ is the unramified strongly cyclic cover corresponding to $\beta$, and every strongly cyclic cover occurs in this way,
	\item two strongly cyclic covers $\tilde{C}_{\beta_1}\to C$ and $\tilde{C}_{\beta_2}\to C$ are isomorphic if and only if $\beta_1+\beta_2$ is a multiple of $\alpha$, in which case the change of coordinates is:\[(x,y,z)\mapsto \left(x,y,\zeta_d \frac{y^j}{z}\right),\] where $\zeta_d$ is a primitive $d$th root of unity,
	\item Let $A_\beta=\{x_i|\beta(x_i)\neq 0\}$.  Then $\tilde{C}_\beta$ is a ramified $d$-cyclic cover of a curve of genus $g_0=\frac{(d-1)(|A_\beta|-2)}{2}$ and curves of genera $g_i=\frac{(d-1)(|A_{i\alpha-\beta}|-2)}{2}$ where $i=1,\ldots,d-1$.
	\item The number of strongly cyclic covers of $C_\beta$ that cover a curve of genus $g=\frac{(k-2)(d-1)}{2}$ is $\binom{r}{k}\left(\left(1-\frac{1}{d}\right)^k d^{k-1}-\frac{(-1)^k}{d}\right)$.
\end{enumerate}
\end{theorem}

We note that in part 2, although not isomorphic as covers of $C$, $\tilde{C}_{\beta_1}$ and $\tilde{C}_{\beta_2}$ are isomorphic as curves if $\beta_2$ is a multiple of $\beta_1$.  Specifically, if $\mu$ is the line bundle corresponding to $\beta_1$, then they are given by the relative spectra $\sh{O}_C\oplus\mu\oplus\ldots\oplus\mu^{d-1}$, and $\sh{O}_C\oplus\mu^k\oplus\ldots\oplus\mu^{(d-1)k}$, which are isomorphic because $d$ and $k$ are relatively prime, and $d$ is the order of $\mu$.

For part 4, the result of Bujalance \cite{MR855143} states the $d=2$ case of this slightly more strongly, as a partition over half of the genera that appear.  For higher degree, the situation is much less clear, just as it is in part 3, because the relation between $\beta$ and the $i\alpha-\beta$'s is more complex.

\begin{proof}
To prove the first point, it suffices to show that if $f_i:C_i\to\PP^1$ for $i=1,2$ be maps of degree $d$ with the same Galois group.  Assume further that $\Branch(f_2)\subset\Branch(f_1)$ and the ramification type over each common branch point is the same.  Then the normalization of $C_1\times_{\PP^1} C_2\to C_2$ is unramified.

This is a local issue near a branch point.  So, without loss of generality, we can take $f_i:\CC\to\CC$ to be $z^k$ and $w^k$ for some $k$.  Then, the fiber product is the curve $z^k=w^k$ in $\CC^2$ and the structure map to $\CC$ is $z^k$.  The fiber over $0$ is singular, but it has $k$ distinct branches, so that the normalization is unramified over $0$.

Part two follows by a direct computation, where $f(x)$ is defined by $\frac{z_1z_2}{y}=f(x)$ and $z_1,z_2$ are the $z$-coordinates for the covers given by $\beta_1$ and $\beta_2$.

Now, we show that this is the only case.  Let $\phi:\tilde{C}_{\beta_1}\to\tilde{C}_{\beta_2}$ be an isomorphism over $C$.  The field $\CC(\tilde{C}_{\beta_1})$ is generated by $x,y,z$, so we'll denote it by $\CC(x,y,z)$ and ignore the relations in the notation.  So there is a rational function $R(x,y,z)\in\CC(x,y,z)$ so that $\phi(x,y,z)=(x,y,R(x,y,z))$, and $R(x,y,z)^d=\prod_{i=1}^r (x-b_i)^{\beta_2(x_i)}$.  Under the action of multiplying $y$ and $z$ by $d$th roots of unity, the fixed fields of index $d$ are $\CC(x,y)=\CC(C_{\beta_1})$, $\CC(x,z)=\CC(C_{\beta_2})$, and $\CC(x,y^i/z)=\CC(C_{i\alpha-\beta})$.  Let $w$ be $y,z$ or $y^i/z$ for some $i$.

Then, we can write $R(x,y,z)=a(x)+b(x)w$.  Then, as $R(x,y,z)^d\in\CC(x)$, we have $a(x)^d+b(x)^dw^d+a(x)b(x)w(\sum_{i=0}^{d-2}f_i(x)w^i)\in\CC(x)$.  So then either $w\in\CC(x)$ or $a(x)b(x)=0$, and if $b(x)=0$, then $R(x,y,z)\in \CC(x)$, thus, $a(x)=0$.  So $R=bw$.  Raising to the $d$th power, we get that $\prod_{i=1}^r (x-b_i)^{\beta_2(x_i)}=b^d w^d=b^d\prod_{i=1}^r (x-b_i)^{\chi(x_i)}$ where $\chi$ is $\beta_1$,$\beta_2$ or $i\alpha-\beta_1$, depending on $w$.  We divide and get $b^d=\prod_{i=1}^r (x-b_i)^{\beta_2(x_i)-\chi(x_i)}$, and as $|\beta_2(x_i)-\chi(x_i)|<d$, we must have $b(x)^d=1$, so $b(x)=\zeta_d$ for some root of unity.  Thus, $\beta_2=\chi$, and we have $\beta_2=i\alpha-\beta_1$ for some $i$.

Part three is just the map $\tilde{C}_{\beta}\to C_\beta$ combined with part two.

Part four is a fairly straightforward counting argument.  We need to have exactly $k$ ramification points, we need to count the number of $\beta$'s such that $|A_\beta|=k$.  This is $\binom{r}{k}$ times the number of linear functions $\{x_1,\ldots,x_k\}\to \FF_d$ whose values add up to zero and none of them are zero.  This can be counted by inclusion-exclusion, giving us $\sum_{i=0}^k (-1)^i \binom{k}{i}d^{k-i-1}$, and this simplifies to $\left(1-\frac{1}{d}\right)^k d^{k-1}-\frac{(-1)^k}{d}$, giving us the count.
\end{proof}

\begin{corollary}
Any degree $d$ cyclic curve given by $y^d=f(x)$ such that $f(x)\in\QQ[x]$ and which splits over $\QQ$ into the product of two polynomials of degrees divisible by $d$, then the curve admits a smooth strongly cyclic cover defined over $\QQ$.
\end{corollary}

Finally, we look at the branch points of those strongly cyclic covers which are also $d$-cyclic curves.  These are given by pairs of branch points, $b_i,b_j$ and we require that the corresponding values of $\beta$ are $\beta(x_i)=1$ and $\beta(x_j)=d-1$.  Then, the cover is given by the equations \[y^d=\prod_{k=1}^r (x-b_k)^{\beta(x_k)}\quad z^d=(x-b_i)(x-b_j)^{d-1}.\]  We make the substitution $z=t(x-a_j)$, which transforms the equation in $z$ to $t^d=\frac{x-a_i}{x-a_j}$, which we can then solve for $x$ to get $x=\frac{a_i-a_j t^d}{1-t^d}$.

From here, it becomes a straightforward manipulation:

\begin{eqnarray*}
y^d&=&\prod_{k=1}^r \left(\frac{a_i-a_jt^d}{1-t^d}-a_k\right)^{\beta(x_k)}\\
&=&\prod_{k=1}^r \left(\frac{t^d(a_k-a_j)-(a_k-a_i)}{1-t^d}\right)^{\beta(x_k)}\\
&=&\frac{t^d(a_i-a_j)^{d}}{(1-t^d)^d}\prod_{k\neq i,j} \left(\frac{t^d(a_k-a_j)-(a_k-a_i)}{1-t^d}\right)^{\beta(x_k)}\\
&=&\frac{t^d(a_i-a_j)^{d}}{(1-t^d)^{\sum_k \beta(x_k)}}\prod_{k\neq i,j} \left(t^d(a_k-a_j)-(a_k-a_i)\right)^{\beta(x_k)}\\
&=&\frac{t^d(a_i-a_j)^{d}}{(1-t^d)^{\sum_k \beta(x_k)}}\prod_{k\neq i,j}(a_k-a_j)^{\beta(x_k)}\prod_{k\neq i,j} \left(t^d-\frac{a_k-a_i}{a_k-a_j}\right)^{\beta(x_k)}\\
\end{eqnarray*}

Now, we set \[w=\frac{y(1-t^d)^{(\sum \beta(x_k))/d}}{t(a_i-a_j)\left(\prod_{k\neq i,j}\left(a_k-a_j\right)^{\beta(x_k)}\right)^{1/d}}\] and obtain the equation \[w^d=\prod_{k\neq i,j} \left(t^2-\frac{a_k-a_i}{a_k-a_j}\right)^{\beta(x_k)}\]

\begin{corollary}
Let $C$ be a $d$-cyclic curve with equation $y^d=\prod_{k=1}^r (x-a_i)^{\beta(x_i)}$.  Then for each ordered pair (except when $d=2$, then unordered) of point $(a_i,0)$ and $(a_j,0)$, there is a smooth strongly cyclic cover, with additional equation $z^d=(x-a_i)(x-a_j)^{d-1}$, $F_{i,j}:C_{i,j}\to C$ where $C_{i,j}$ is itself $d$-cyclic, given by $y^d=\prod_{k\neq i,j} \left(x^2-\frac{a_k-a_i}{a_k-a_j}\right)^{\beta(x_k)}$ with \[F_{i,j}(x,y)=\left(\frac{a_i-a_j x^d}{1-x^d},\frac{wt(a_i-a_j)\left(\prod_{k\neq i,j}(a_k-a_j)\right)^{1/d}}{(1-t^d)^{\sum \beta(x_j)/d}}\right).\]  This can be characterized, up to equivalence, by the fact that the two chosen points are the only ones not covered by ramification points.
\end{corollary}

\bibliographystyle{alpha}
\bibliography{Cyclic}
\end{document}